\documentclass[a4paper]{article}

\usepackage[utf8]{inputenc}
\usepackage{lmodern}
\usepackage[T1]{fontenc}

\usepackage{mathtools, amssymb, amsthm}
\usepackage{enumitem}
\usepackage{bbm,dsfont}
\numberwithin{equation}{section}

\newcommand{\MR}{\textit{MR}}

\newcommand{\R}{\mathds{R}}
\newcommand{\C}{\mathds{C}}

\newcommand{\A}{\mathcal{A}}
\newcommand{\T}{\mathcal{T}}

\renewcommand{\L}{\mathcal{L}}

\renewcommand{\P}{\mathcal{P}}
\renewcommand{\Re}{\operatorname{Re}}

\newcommand{\fra}{\mathfrak{a}}

\renewcommand{\mid}{\, \vert \,}

\DeclarePairedDelimiter\abs{\lvert}{\rvert}
 \DeclarePairedDelimiter\norm{\lVert}{\rVert}

\theoremstyle{plain}
\newtheorem{theorem}{Theorem}[section]

\newtheorem{lemma}[theorem]{Lemma}

\theoremstyle{definition}

\newtheorem{problem}[theorem]{Problem}

\newtheorem{remark}[theorem]{Remark}

\begin{document}
\title{ Evolution Equations governed by Lipschitz Continuous Non-autonomous Forms \footnote {Work  partly supported by DFG (JA 735/8-1)}}
\author{
    Ahmed Sani and Hafida Laasri
}


\maketitle

\begin{abstract}\label{abstract}
We prove $L^2$-maximal regularity  of linear non-autonomous evolutionary Cauchy problem
\begin{equation}\label{eq00}\nonumber
\dot{u} (t)+A(t)u(t)=f(t) \hbox{ for }\  \hbox{a.e. t}\in [0,T],\quad u(0)=u_0,
\end{equation}
where the operator  $A(t)$ arises from   a time depending sesquilinear form $\fra(t,.,.)$ on a Hilbert space $H$ with constant  domain $V.$ We prove the maximal regularity in $H$ when these forms are time Lipschitz continuous. We proceed by approximating the problem using the frozen coefficient method  developed in \cite{ELKELA11}, \cite{ELLA13} and \cite{LH}. As a consequence, we obtain an invariance criterion for convex and closed sets of $H.$

\end{abstract}

\medskip


\section{Introduction}\label{section:introduction}
In this paper we study non-autonomous evolutionary linear Cauchy-problems
\begin{equation}\label{eq0}
\dot{u} (t)+\A(t)u(t)=f(t), \quad u(0)=u_0,
\end{equation}
where the operators $\A(t), \ t\in[0,T],$ arise from  sesquilinear forms on Hilbert spaces. More precisely, throughout this work
 $H$ and $V$ are two separable Hilbert spaces. The   scalar  products and the corresponding norms on $H$ and $V$ will be denoted by  $(. \mid .)$, $(. \mid .)_V$, $\norm{.}$ and $\norm{.}_V$, respectively. We assume that  $V \underset d \hookrightarrow H;$ i.e., $V$ is  densely embedded into $H$ and 
\begin{equation}\label{eq:V_dense_in_H}
    \norm{u} \le c_H \norm u _V \quad (u \in V)
\end{equation}
for some constant $c_H>0.$ \\
Let $V'$ denote  the antidual of $V.$  The duality
between $V'$ and $V$ is denoted by $\langle ., . \rangle$. As usual, we identify $H$ with  $H'.$ It follows that $V\hookrightarrow H\cong H'\hookrightarrow V'$ and so $V$ is identified with a subspace of $V'.$ These embeddings are continuous and
\begin{equation}\label{eq:H_dense_in_V'}
    \norm{f}_{V'} \le c_H \norm f  \quad (f \in V')
\end{equation}
with the same constant $c_H$ as in (\ref{eq:V_dense_in_H}) (see e.g., \cite{Bre11}).\\
For a non-autonomous form
\[
	\fra: [0,T]\times V\times V \to \C
\]
such that   $\fra(t, .,.)$ is sesquilinear for all $t\in[0,T]$, $\fra(.,u,v)$ is measurable for all $u,v\in V,$
\begin{equation*}\label{eq:continuity-nonaut}
	\abs{ \fra(t,u,v) } \le M \norm{u}_V \norm{v}_V \quad (t\in[0,T],u,v\in V)\qquad
\end{equation*}
and
\begin{equation*}\label{eq:Ellipticity-nonaut}
	\Re \fra (t,u,u) +\omega\norm{u}^2\ge \alpha \|u\|^2_V \quad ( t\in [0,T], u\in V)
\end{equation*}
for some  $\alpha>0, M\geq 0$ and $\omega\in \R,$ for each $t\in[0,T]$ we can associate a unique operator $\A(t)\in \L(V,V')$  such that
$\fra(t,u,v)=\langle\A(t)u,v\rangle \hbox{ for all } u,v\in V.$ It is a known fact
 that $-\A(t)$ with domain $V$ generates a holomorphic semigroup $(\mathcal T_t(s))_{s\geq 0}$ on $V'$. Observe that $\|\mathcal A(t)u\|_{V'}\leqslant M \|u\|_V$ for all $u\in V$ and all $t\in [0,T]$. It is worth to mention that the mapping $t\mapsto \mathcal A(t)$ is strongly measurable by the Dunford-Pettis Theorem \cite{ABHN11} since the spaces are assumed to be separable and $t\mapsto \mathcal A(t)$ is weakly measurable. Thus $t\mapsto \mathcal A(t)u$ is Bochner integrable on $[0,T]$ with values in $V'$ for all $u\in V.$

\bigskip
 \par The following well known  maximal regularity  result is due to J. L. Lions.

\begin{theorem}\label{wellposedness in V'}
 Given $f\in L^2(0,T;V^\prime)$ and $u_0\in H,$ there is a unique solution $u \in \MR(V,V'):=L^2(0,T;V)\cap H^1(0,T;V')$ of
\begin{equation}\label{nCP in V'}
\dot{u} (t)+\A(t)u(t)=f(t), \quad u(0)=u_0.
\end{equation}
\end{theorem}
\par Note that $MR(V,V')\underset d \hookrightarrow
C([0,T];H)$ (see \cite[p.106]{Sho97}), so the condition $u(0)=u_0$ in (\ref{nCP in V'}) makes sense and the solution  is  continuous on $[0,T]$ with values in $H.$ \par\noindent  The proof of Theorem \ref{wellposedness in V'} can be given by an  application of Lions’ Representation
Theorem \cite{Lio61} (see also \cite[p. 112]{Sho97} and \cite[Chapter 3]{Tho03}) or by Galerkin's method  \cite[ XVIII
Chapter 3, p. 620]{DL88}. We refer also to an alternative   proof given by Tanabe \cite[Section 5.5]{Tan79}.

\par\noindent In Section \ref{s2}, we give an other  proof  by using the approach of frozen coefficient developed  in \cite{ELKELA11}, \cite{ELLA13} and \cite{LH}, from which we derive the criterion for invariance of convex closed sets established by \cite{ADO12} and also the recent result given by \cite{ADLO13} for Lipschitz continuous forms.
\par\noindent Let
$\Lambda:=(0=\lambda_0<\lambda_1<...<\lambda_{n+1}=T)$ be a
subdivision of $[0,T].$ We approximate (\ref{eq0}) by (\ref{nCP in V'0}), obtained when the generators $\A(t)$ are frozen on the interval $[\lambda_k,\lambda_{k+1}[.$  More precisely, let $\A_\Lambda:\ [0,T]\rightarrow
\mathcal{L}(V,V')$ be given by
\[ \A_\Lambda
(t):=\left\{%
\begin{array}{ll}
    \A_k & \hbox{for } \lambda_k\leq t<\lambda_{k+1},\\
    \A_{n} & \hbox{for } t=T, \\
\end{array}%
\right. \]  with   \[\displaystyle \A_kx
:=\frac{1}{\lambda_{k+1}-\lambda_k}
\int_{\lambda_k}^{\lambda_{k+1}}\A(r)u{\rm  d}r\ \ \  (u\in V,
k=0,1,...,n).\]
Note that the integral in the right hand side makes sense since the mapping $t\mapsto \mathcal A(t)$ is, as mentioned above, strongly Bochner-integrable.\\
We show  (see Theorem \ref{alternative-proof}) that for all $u_0\in H$ and $f\in L^2(0,T;V')$  the non-autonomous problem
 \begin{equation}\label{nCP in V'0}
\dot{u}_\Lambda (t)+\A_\Lambda(t)u_\Lambda(t)=f(t), \quad u_\Lambda(0)=u_0
\end{equation}
 \noindent has a unique solution $u_{\Lambda}\in MR(V,V')$ which converges in $MR(V,V')$ as $\vert
\Lambda\vert\rightarrow 0$ and  $u:=\lim\limits_{ |\Lambda| \to 0}u_{\Lambda}$ solves uniquely  $(\ref{nCP in V'}).$
\par Let $\mathcal C$ be a closed convex subset of the Hilbert space $H$ and let $P: H\rightarrow \mathcal{C}$ be the orthogonal projection onto $\mathcal C.$ As a consequence of Theorem \ref{alternative-proof} we obtain: If $u_0\in\mathcal C, P(V)\subset V$ and
\begin{equation}\label{invariance-homo}\Re\fra(t,Pv,v-Pv)\geq 0\end{equation}
for almost every $t\in [0,T]$ and for all $v\in V,$ then $u(t)\in \mathcal C$ for all $t\in [0,T],$ where $u$ is the solution of (\ref{nCP in V'}) with $f=0.$ In the autonomous case condition
(\ref{invariance-homo}) is also necessary for the invariance of $\mathcal C$, see \cite{Ouh96}. More recently, for $f\neq 0$ the invariance of
$\mathcal C$ under the solution of (\ref{nCP in V'}) was proved by Arendt, Dier and Ouhabaz \cite{ADO12} provided that
\begin{equation}\label{invariance-non-homo}
\Re\fra(t,Pv,v-Pv)\geq  \langle f(t),v-Pv\rangle
\end{equation}
for almost every $t\in [0,T]$ and for all $v\in V.$
\par\bigskip  Theorem \ref{wellposedness in V'} establishes $L^2$-maximal regularity of the Cauchy problem (\ref{nCP in V'}) in $V'$ assuming only that  $t\mapsto \fra(t,u,v)$  is  measurable for all $u,v\in V$. However, in applications to  boundary valued problems, only the part $A(t)$ of $\A(t)$ in $H$ does realize the boundary conditions in question.  Thus one is interested in $L^2$-maximal regularity in $H$:
 \begin{problem}\label{Lions's problem}
If $f\in L^2(0,T; H)$ and $u_0\in V$, does the solution of (\ref{nCP in V'0}) belong to $\MR(V,H):=L^2(0,T;V)\cap H^1(0,T;H)$?
\end{problem}
\noindent This Problem \ref{Lions's problem} is asked (for $u_0=0$) by Lions \cite[p.\ 68]{Lio61} and is, to our knowledge, still open.
 Note that if $\fra$ (or equivalently $\A$) is a step function the answer to Problem  \ref{Lions's problem}  is affirmative. In fact, for $u_0\in V$ and $f\in L^2(0,T; H)$ the solution $u_\Lambda$ of  (\ref{nCP in V'0}) belongs to $\MR(V,H)\cap C([0,T]; V)$ (see Section \ref{s2}). Thus, Problem \ref{Lions's problem} can be reformulated as follows:
\begin{problem}\label{Lions's problem+discrétisé}
If $f\in L^2(0,T; H)$ and $u_0\in V$,  does  the solution  of  (\ref{nCP in V'0}) converge in $\MR(V,H)$ as $\vert
\Lambda\vert\rightarrow 0$ ?
\end{problem}
\noindent For general forms, a positive answer of Problem \ref{Lions's problem} is given under additional regularity assumption (with respect to $t$) on
 $\fra(t,.,.).$ For symmetric forms, Lions proved $L^2$-maximal regularity in $H$ for $u_0=0$ ( respectively for $u_0\in D(A(0))$) provided  $\fra(.,u,v) \in C^1 [0,T]$ (respectively $\fra(.,u,v) \in C^2 [0,T]$ ) for all $u,v \in V,$  \cite[p.~68 and p.~94)]{Lio61}. Moreover, a combination of \cite[Theorem~1.1, p.~129]{Lio61} and \cite[Theorem~5.1, p.~138]{Lio61} shows that
if $\fra(., u, v) \in C^1[0,T]$  for all
$u, v \in V$, then (\ref{nCP in V'0}) has $L^2$-maximal regularity in $H.$  Bardos \cite{Bar71}  gave also a positive answer to Problem (\ref{Lions's problem}) under the assumptions that the domains of both $A(t)^{1/2}$ and
$A(t)^{*1/2}$ coincide with $V$ and that $\A(.)^{1/2}$ is continuously differentiable with values in $\L(V,V')$. We mention also a result of Ouhabaz and Spina \cite{OS10} and Ouhabaz  and Haak \cite{OH14}. They proved $L^2$-maximal regularity for (possibly non-symmetric) forms such that
$\fra(.,u,v) \in C^\alpha[0,T]$ for all $u,v \in V$  and some
$\alpha > \frac 1 2$. The result in  \cite{OS10} concerns the case $u_0=0$ and the one in \cite{OH14} concerns the case $u_0$ in the real-interpolation space
$(H,D(A(0)))_{1/2,2}.$
\bigskip
\par In Section $\ref{s3},$ we are concerned with a recent result obtained in \cite{ADLO13}. Assume that the sesquilinear form $\fra$ can be written as
$\fra(t,u,v) = \fra_1(t,u,v) + \fra_2(t,u,v)$
where $\fra_1$ is symmetric, bounded (i.e $\fra_1(t,u,v)\leqslant M_1 \|u\|\|v\|$, $M_1\geq 0$) and coercive as above and piecewise Lipschitz-continuous on $[0,T]$ with Lipschitz constant $L_1$,
whereas $\fra_2\colon [0,T] \times V \times H \to \C$ satisfies $\abs{\fra_2(t,u,v)}\le M_2 \norm u_V \norm v_H$
and $\fra_2(.,u,v)$ is measurable for all $u\in V$, $v \in H$.
Furthermore, let
$B\colon [0,T] \to \L(H)$ be strongly measurable such that $\norm{B(t)}_{\L(H)} \le \beta_1$ for all $t \in [0,T]$ and
$0 < \beta_0 \le (B(t)g \mid g)_H$ for $g \in H$, $\norm{g}_H=1$, $t \in [0,T].$ Then, the following result is proved in \cite[Corollary 4.3]{ADLO13} :

\begin{theorem}\label{thm: ADOL}
    Let $u_0 \in V$, $f \in L^2(0,T;H)$. Then there exists a unique $MR(V,H)$ satisfying
    \begin{equation*}
        \dot u(t) + B(t)\A(t)u(t) = f(t) \quad \text{a.e.}\quad
                        u(0)=u_0.
    \end{equation*}
    Moreover
    \begin{equation}\label{eq:MR_estimate}
    	\norm u_{\MR(V,H)} \le C \Big[ \norm{u_0}_V + \norm f_{L^2(0,T;H)} \Big],
    \end{equation}
where the constant $C$ depends only on $\beta_0,\beta_1,M_1,M_2,\alpha,T,L_1$ and $\gamma$.
     \end{theorem}
\noindent In the special case where $B=I$ and $\fra=\fra_1$ (or equivalently $\fra_2=0$) we proof that  Problem \ref{Lions's problem+discrétisé} has a positive answer.
\par\noindent We emphasize that our result on approximation  may be applied to  concrete linear evolution equations.  For example, to evolution equation governed by elliptic operator in nondivergence form on a domain $\Omega$ with time depending coefficients
 $$
\left\{
     \begin{aligned}
      \dot {u}(t)&  - \sum_{i,j} \partial_i a_{ij}(t,.) \partial_j u(t)  = f(t)
     \\ u(0) &=u_0 \in H^1(\Omega) .
\end{aligned} \right.
$$
with an appropriate Lipschitz continuity property on the coefficients with respect to $t$ and boundary conditions such as Neumann or non-autonomous Robin boundary conditions.
\subsection*{Acknowledgment}
 The authors are most grateful to Wolfgang Arendt and Omar El-Mennaoui for fruitful discussions on  maximal regularity and invariance criterion for the non-autonomous linear Cauchy problem.
\section{Preliminary}\label{s1}

 Consider a  \emph{continuous} and $H$-\emph{elliptic} sesquilinear form $\fra:V \times V \to \C$. This means,
respectively
\begin{equation}\label{eq:a_continuous}
    \abs{\fra(u,v)} \le M \norm u _V \norm v _V \quad \hbox{ for some }  M\geq 0 \hbox{ and all } u,v \in V,
\end{equation}
\begin{equation}\label{eq:H-elliptic}
    \Re \fra(u) + \omega \norm u^2 \ge \alpha \norm u _V^2 \quad \hbox{ for some   } \alpha>0,~\omega\in \R\hbox{ and all } u \in V.
\end{equation}
Here and in the following we shortly write $\fra(u)$ for $\fra(u,u).$ The operator $\A \in \L(V,V')$  associated with  $\fra$  on $V'$ is defined by
\[
\langle \A u, v \rangle = \fra(u,v) \quad (u,v \in V).
\]

\par\noindent Seen as an unbounded operator on $V'$ with domain $D(\A) = V,$ the
operator $- \A$ generates a holomorphic $C_0-$semigroup $\T$ on $V'.$ The semigroup is bounded on a  sector if $\omega
=0$, in which case $\A$ is an isomorphism. Denote by $A$ the part
of $\A$ on $H$; i.e.,\
\begin{align*}
    D(A) := {}& \{ u\in V : \A u \in H \}\\
    A u = {}& \A u.
\end{align*}
It is a known fact that  $-A$ generates a holomorphic $C_0$-semigroup  $T$ on $H$  and $T=\mathcal T_{\mid H}$ is the
restriction of the semigroup generated by $-\A$ to $H.$ Then $A$ is the operator \textit{induced} by $\fra$  on $H.$
 We refer to \cite{Ka},\cite{Ouh05} and  \cite[Chap.\ 2]{Tan79}.
 {\remark\label{remark-rescaling} The sesquilinear form $\fra$ satisfies condition (\ref{eq:H-elliptic}) if and only if the form $a_\omega$ given by
 \[\fra_\omega(u,v):=\fra(u,v)+\omega (u\mid v)\]
 is coercive. Moreover, if $\mathcal T_\omega$ (respectively $\mathcal A_{\omega}$) denotes the semigroup (respectively the operator) associated with $\fra_\omega,$ then $\T_\omega(t)=e^{-\omega t}\mathcal T(t)$ and $\mathcal A_{\omega}=\omega +\mathcal A$ for all $t\geq 0.$ Then it is possible to choose, without loss of generality, $\fra$ coercive (i.e., $\omega=0.$)}

 \par  The following maximal regularity results are well known: If
 $u_0\in H, f\in L^2(a,b;V')$ then the function \[
u(t)=\T(t)u_0+\int_a^t \T(t-r)f(r){\rm d}r\]
belongs to $L^2(a,b;V)\cap H^1(a,b;V')$  and is the unique solution of  the autonomous  initial
value problem
\begin{equation}\label{ACP in V'}
\dot{u} (t)+\A u(t)=f(t),\qquad \mbox{t.a.e on } [a,b]\subset[0,T], \quad u(a)=u_0.
\end{equation}
\par\noindent Recall that the maximal regularity space
\begin{equation}\label{MR(V,V')}\MR (a,b;V,V'):=L^2(a,b;V)\cap H^1(a,b;V')\end{equation}
is continuously embedded in $C([a,b],H)$ and  if  $u\in \MR (a,b;V,V')$ then
the function  $\norm{u(.)}^2$ is absolutely continuous on $[a,b]$ and
\begin{equation}\label{eq:chain_rule V'}
\frac{d}{dt}\norm{ u(.)}^2=2\Re\langle\dot u,u\rangle
\end{equation}
see e.g., \cite[Chapter III, Proposition 1.2]{Sho97} or \cite[Lemma 5.5.1]{Tan79}. For $[a,b]=[0,T]$ we shortly denote $MR(a,b;V,V')$ by $MR(V,V')$ in (\ref{MR(V,V')}).

\par\noindent
Furthermore, if $(f,u_0)\in L^2(a,b;H)\times V$ then the solution $u$  of (\ref{ACP in V'}) belongs to the  maximal regularity space
\begin{equation}\label{MR(H,D)}
MR(a,b ;D(A),H):=L^2(a,b ;D(A))\cap H^1(a,b;H)
\end{equation}
which is equipped with the norm $\|.\|_{MR}$ given for all $u\in MR(a,b;D(A),H)$ by
\begin{equation}\label{MR(H,D)-Norm}
\|u\|_{MR}^2:=\int_a^b\|u(t)\|^2dt +\int_a^b\|\dot u(t)\|^2dt +\int_a^b\|Au(t)\|^2dt.
\end{equation}
\par\noindent The maximal regularity space $MR(a,b; D(A), H)$ is continuously embedded into $C([a,b]; V),$ \cite[Exemple 1, page 577]{DL88}.
If  the form $\fra$ is symmetric,  then for each  $u\in MR(a,b; D(A),H),$
 the function $\fra(u(.))$ belongs to $W^{1,1}(a,b)$ and the following product formula holds
\begin{equation}\label{rule-formula}
\frac{d}{dt}\fra(u(t))=2(Au(t)\mid \dot u(t)) \hbox{  for a.e. } t\in [a,b],
\end{equation}
for the proof  we refer to \cite[Lemma 3.1]{AC10}.
\par\bigskip
 The following lemma gives a locally uniform estimate for the solution of the autonomous problem. This estimate will play an important role in the study of the convergence in Theorem \ref{(nCP) in H}.
{\lemma\label{indepmax}\cite[Theorem 3.1]{AC10} Let  $\fra$ be a continuous and $H$-elliptic sesquilinear form. Assume the form $\fra$ is symmetric.  Let $f\in
L^2(a,b;H)$ and $u_0\in V.$ Let  $u\in MR(a,b; D(A),H)$ be such that
\begin{equation}\label{ACP in H}
\dot{u} (t)+A u(t)=f(t),\qquad t.a.e\  on \ [a,b]\subset[0,T], \quad u(a)=u_0.
\end{equation}
Then  there exists a constant
$c_1>0$ such that
\begin{equation}\label{maxregconst2}
    \sup_{s\in[a,b]}\|u(s)\|_{V}^2\leq c_1\Big[\|u(a)\|_V^2+\|f\|_{L^2(a,b; H)}^2\Big]
\end{equation}
where $c_1=c_1(M, \alpha, \omega, T)>0$ is independent of $f, u_0$ and  $[a,b]\subset[0,T].$ }
\par\bigskip For the sake of completeness, we include here  a simpler proof in the non restrictive case $\omega=0.$

\begin{proof} We use the same technique as in the proof of \cite[Theorem 3.1]{AC10}. For simplicity and according to Remark \ref{remark-rescaling} we may assume without loss of generality that $\omega=0$ in (\ref{eq:H-elliptic}).   For almost every
$t\in[a,b]$ \[ (\dot{u}(t)\mid\dot{u}(t))+
(Au(t)\mid\dot{u}(t))=(f(t)\mid\dot{u}(t)).\]
The rule formula (\ref{rule-formula}) and the Cauchy-Schwartz inequality together
with the Young inequality  applied to the term on the right-hand
side of the above equality imply that, for almost every
$t\in[a,b]$
\[\frac{1}{2}\|\dot{u}(t)\|^2+\frac{1}{2}\frac{d}{dt}\fra(u(t))
\leq\frac{1}{2}\|f(t)\|^2.\] Integrating this
inequality on $[a,t],$ it follows that
\[ \int_a^t \|\dot{u}(s)\|^2ds+ \fra(u(t))\leq \fra(u(a))+ \int_a^t
\|f(s)\|^2ds.\] Thus, by (\ref{eq:a_continuous}) and
(\ref{eq:H-elliptic}),
\begin{equation}\label{eq1}
\int_a^t \|\dot{u}(s)\|^2ds+\alpha\|u(t)\|_V^2\leq
M\|u(a)\|_{V}^2+\|f\|^2_{L^2(a,b;H)}
\end{equation}
for almost every $t\in [0,T].$ It follows that
\begin{equation}\label{eq4}
    \sup_{t\in[a,b]}\|u(t)\|_V^2\leq \frac{1}{\alpha}\Big(M\|u(a)\|_V^2+\|f\|^2_{L^2(a,b;H)}\Big)
\end{equation}
which gives the desired estimate.
\end{proof}
 {\remark Lemma \ref{indepmax}  says that the constant $c_1$ in  (\ref{eq4}) depends only on $M,\alpha,\omega$ and $T,$ but it does not depend on the subinterval $[a,b]$ or on other properties of $\fra$.}

\section{Well-posedness in $V'$}\label{s2}
Let $H,V$ be the Hilbert spaces explained  in the previous sections. Let $T>0$ and let
\[
	\fra: [0,T]\times V\times V \to \C
\]
be a \textit{non-autonomous form}, i.e., $\fra(t, .,.)$ is sesquilinear for all $t\in[0,T]$, $\fra(.,u,v)$ is measurable for all $u,v\in V,$
\begin{equation}\label{eq:continuity-nonaut}
	\abs{ \fra(t,u,v) } \le M \norm{u}_V \norm{v}_V \quad (t\in[0,T],u,v\in V)\qquad
\end{equation}
and
\begin{equation}\label{eq:Ellipticity-nonaut}
	\Re \fra (t,u,u) +\omega\norm{u}\ge \alpha \|u\|^2_V \quad ( t\in [0,T], u\in V)
\end{equation}
for some  $\alpha>0, M\geq 0$ and $\omega\in \R.$
 \\\par\noindent  We recall that, for all $t\in[0,T]$ we denote by $\A(t)\in \L(V,V')$  the operator associated with the form $\fra(t,.,.)$ in $V'$  and by $\T_t$ the analytic $C_0$-semigroup generated by $-\A(t)$ on $V'.$  Consider the non-autonomous Cauchy problem
 \begin{equation}\label{nCP in V'2}
\dot{u} (t)+\A(t) u(t)=f(t),\quad \hbox{ for a.e }  t\in [0,T], \quad u(0)=u_0.
\end{equation}
In this section, we are interested in the well-posedness of (\ref{nCP in V'2}) in $V'$ with $L^2$-maximal regularity.
 The case where  $\fra$ is independent on
$t$ is described in the previous section.
 \par\noindent The case where $\fra$ is a step function is also easy to
describe. In fact, let $\Lambda=(0=\lambda_0<\lambda_1<...<\lambda_{n+1}=T)$ be
a subdivision of $[0,T].$ Let \[\fra_k:V \times V \to \C\ \ \hbox{ for } k=0,1,...,n\] a finite family of continuous and $H$-elliptic forms.
The associated operators are denoted by $\A_k\in \L(V,V').$ Let $\T_k$ denote the $C_0-$semigroup generated by $-\A_k$ on $V'$ for all $k=0,1...n.$  The function
\begin{equation}\label{formlambda}\fra_\Lambda:[0,T]\times V \times V \to \C\end{equation}
defined by   $\fra_\Lambda(t;u,v):=\fra_k(u,v)$ for $\lambda_k\leq t<\lambda_{k+1}$ and $\fra_\Lambda(T;u,v):=\fra_{n}(u,v),$ is strongly  measurable on $[0,T]$. Let \[\A_\Lambda:[0,T]\to \L(V,V')\] be given by $\A_\Lambda
(t):=\A_k$ for  $\lambda_k\leq t<\lambda_{k+1},\ k=0,1,...,n,$ and $\A_\Lambda (T):=\A_{n}.$   For each subinterval
$[a,b]\subset [0,T]$ such that
 \[\lambda_{m-1}\leq a<\lambda_m<...<\lambda_{l-1}\leq b<\lambda_{l}\]
 we define the operators $\P_\Lambda (a,b)\in \mathcal{L}(V')$ by
  \begin{equation}\label{promenade1}\P_\Lambda (a,b):= \T_{l-1}(b-\lambda_{l-1})
 \T_{l-2}(\lambda_{l-1}-\lambda_{l-2})...\T_{m}(\lambda_{m+1}-\lambda_{m})\T_{m-1}(\lambda_{m}-a),
  \end{equation}
and for  $\lambda_{l-1}\leq a\leq b<\lambda_{l}$ by
\begin{equation}\label{promenade2}\P_\Lambda (a,b):= \T_{l-1}(b-a).\end{equation}
 It is easy to see, that for all $u_0\in H$ and $f\in L^2(a,b, V')$ the function
 \begin{equation}\label{prom-sol-no-homogen} u_\Lambda(t)=\P_\Lambda(a,t)u_0+\int_a^t\P_\Lambda(r,t)f(r){\rm d}r\end{equation}
belongs to $\MR(a,b;V,V')$ and is the unique solution of  the initial
value problem
\begin{equation*}
\dot{u}_\Lambda(t)+\A_{\Lambda}(t)u_\Lambda(t)=f(t),\quad \hbox{ for a.e }  t\in [a,b]\subset[0,T], \quad u_\Lambda(a)=u_0.
\end{equation*}
The product given by (\ref{promenade1})-(\ref{promenade2}) and also
the existence of a limit of this  product as $|\Lambda|$ converges
to $0$ uniformly on $[a,b]\subset [0,T],$ was studied in
\cite{ELKELA11},\cite{LH} and \cite{ELLA13}.
This leads to a theory of integral product,
comparable to that of the classical Riemann integral. The notion
of product integral has been introduced by V. Volterra at the
end of 19th century. We refer to A.
Slav\'ik \cite{S} and the references therein for a
discussion on the work of Volterra and for more details  on
product integration theory.
\par\bigskip Consider now the general case where $\fra: \ [0,T]\times V\times V\rightarrow
\C$ is a  non-autonomous form and let $\A(t)\in \L(V,V')$ be the associated operator with $\fra(t,.,.)$ on $V'.$ We want to
approximate $\fra$ and $\A$ by step functions. Let  $\Lambda:=(0=\lambda_0<\lambda_1<...<\lambda_{n+1}=T)$  be a subdivision of
$[0,T]$ and $a_\Lambda:\ [0,T]\times V\times V\rightarrow \C $ and $\A_\Lambda:\ [0,T]\rightarrow \mathcal{L}(V,V')$ be as above
  where $\A_k$ are associated with the sesquilinear forms
\begin{equation}\label{eq:form-moyen integrale}
\begin{aligned}
\ \fra_k(u,v):=\frac{1}{\lambda_{k+1}-\lambda_k}
\int_{\lambda_k}^{\lambda_{k+1}}&\fra(r;u,v){\rm  d}r\ \ \  \\& \hbox{ for } u,v\in V, \
k=0,1,...,n.
\end{aligned}
\end{equation}
Note that $\fra_k$ satisfies (\ref{eq:continuity-nonaut}) and (\ref{eq:Ellipticity-nonaut}), $k=0,1,...n$, we then have for all $u\in V$
\begin{equation}\label{eq:op-moyen integrale}
 \A_ku :=\frac{1}{\lambda_{k+1}-\lambda_k}
\int_{\lambda_k}^{\lambda_{k+1}}\A(r)u{\rm  d}r.
\end{equation}
\par\noindent Let $u_0\in H$ and $f\in L^2(0,T;V')$ and let $u_\Lambda\in \MR(V,V')$ denote  the unique solution of
\begin{equation}\label{problem-discrétisé}
  \dot{u}_\Lambda(t)+\mathcal A_\Lambda(t)u_\Lambda(t)=f(t),\quad \hbox{ for a.e }  t\in [0,T],\
 \ \ \ \ u_\Lambda(0)=u_0 \
\end{equation}
Recall that $u_\Lambda$ is given explicitly by (\ref{promenade1})-(\ref{prom-sol-no-homogen}). \par\noindent  For simplicity and according to Remark \ref{remark-rescaling}, we may assume without loss of generality that  $\omega=0$ in (\ref{eq:Ellipticity-nonaut}). In fact, let $u_\Lambda\in MR(V,V')$ and $v_\Lambda(t):=e^{-wt}u_\Lambda(t).$ Then $u_\Lambda$ satisfies (\ref{problem-discrétisé}) if and only if $v_\Lambda$ satisfies
\begin{equation}\label{problem-discrétisétranslate}
  \dot{v}_\Lambda(t)+(\omega+\mathcal A_\Lambda(t))v_\Lambda(t)=e^{-wt}f(t) \ \ \  t{\rm
-a.e.}  \hbox{ on} \ [0,T],\
 \ \ \ \ v_\Lambda(0)=u_0 \
\end{equation}
 In the sequel, $\omega=0$ will be our assumption.
\begin{lemma}\label{lem3} Let $u_0\in H$ and $f\in L^2(0,T;V').$ Let $u_\Lambda\in \MR(V,V')$ be the  solution  of
(\ref{problem-discrétisé}). Then there exists a constant $c_2>0$ independent of $f, u_0$ and $\Lambda$ such that
\begin{equation}\label{estim-lem3}
\int_0^t\|u_\Lambda(s)\|_V^2ds\leq
c_2\Big[ \int_0^t\|f(s)\|_{V'}^2ds+\|u_0\|^2\Big],\end{equation}  for a.e   $t\in [0,T].$
\end{lemma}
\begin{proof} Since $u_\Lambda\in MR(V,V'),$ it follows  from (\ref{eq:chain_rule V'})
\begin{align*}\frac{d}{dt}\|u_\Lambda(t)\|^2&=2\Re \langle\dot{u}_\Lambda(t),u_\Lambda(t)\rangle
\\&=2 \Re \langle f(t)-\mathcal{A}_\Lambda(t)u_\Lambda(t),u_\Lambda(t)\rangle
\\&=-2\Re a_\Lambda(t,u_\Lambda(t),u_\Lambda(t))+2 \Re \langle f(t),u_\Lambda(t)\rangle
\end{align*}
for almost every $t\in [0,T].$
Integrating this equality  on $(0,t),$ by coercivity   of the form $\fra$ and the Cauchy-Schwartz inequality  we obtain
\[\|u_\Lambda(t)\|^2+2\alpha\int_0^t\norm{u_\Lambda(s)}_V^2ds\leq 2\int_0^t
\norm{f(s)}_{V'}\norm{u_\Lambda(s)}_Vds+  \|u_0\|^2.\]
 Inequality (\ref{estim-lem3}) follows from this estimate and the standard inequality \[ab\leq \frac{1}{2}(\frac {a^2}{\varepsilon}+\varepsilon b^2) \ \ (\varepsilon>0,\ a,b\in \mathbb{{R}}).\]

\end{proof}
Let $|\Lambda|:=\displaystyle\max_{j=0,1,...,n}(\lambda_{j+1}-\lambda_{j})$ denote the mesh of the subdivision $\Lambda$ of $[0,T].$  The main result of this section is the following
{\theorem\label{alternative-proof} Let $f\in L^2(0,T;V')$ and $u_0\in H.$ Then the  solution $u_\Lambda$
 of (\ref{problem-discrétisé})
  converges weakly in $MR(V,V')$ as $|\Lambda|
  \longrightarrow 0$ and $u:=\lim\limits_{|\Lambda|\to 0}u_\Lambda$
is the unique solution of (\ref{nCP in V'}). }
\begin{proof}
To prove that $\lim u_{\Lambda}$ exists as $|\Lambda|\longrightarrow 0,$ it suffices, by the compactness of bounded sets of $L^2(0,T,V),$ to show that it exists $u\in MR(V,V')$ such that every convergent subsequence of $u_{\Lambda}$ converges to $u$. We then begin with the uniqueness. \textit{Uniqueness:} Let $u\in \MR(V,V')$ be a solution of (\ref{nCP in V'}) with $f=0$ and $u(0)=0.$ Then
\begin{align*}\frac{d}{dt}\|u(t)\|^2&=2\textrm{Re
}\langle\dot{u}(t),u(t)\rangle
\\&=-2 \Re \langle\mathcal{A}(t)u(t),u(t)\rangle
\\&=-2\Re a(t,u(t),u(t)).
\end{align*}
Hence \[\frac{d}{dt}\|u(t)\|^2\leq -2\alpha\|u(t)\|_V^2\] and since
$u(0)=0,$ it follows that  $u(t)=0$ for a.e. $t\in[0,T].$
\par\noindent\textit{Existence:} Let $u_0\in H$ and $f\in L^2(0,T;V').$  Let $u_\Lambda\in \MR(V,V')$ be the solution  of
(\ref{problem-discrétisé}). Since $u_\Lambda$ is bounded in $L^2(0,T;V)$  by Lemma \ref{lem3}, we can assume
(after passing to a subsequence) that $u_\Lambda\rightharpoonup u$ in  $L^2(0,T;V)$  as $\mid\Lambda\mid$ goes to $0$. Let now $g\in L^2(0,T;V).$ We have
 $\mathcal{A}_\Lambda^* g\longrightarrow
\mathcal{A}^* g$  in $L^2(0,T;V')$ \cite[Lemma 2.3 and Lemma 3.1]{ELLA13}. Since
\[\int_0^T\langle \mathcal{A}_\Lambda(s)u_\Lambda(s),g(s)\rangle ds=
\int_0^T\langle u_\Lambda(s),\mathcal{A}^*_\Lambda(s)g(s)\rangle ds,\] it follows that \[\int_0^T\langle \mathcal{A}_\Lambda(s)u_\Lambda(s),g(s)\rangle ds \rightarrow \int_0^T\langle \mathcal{A}(s)u(s),g(s)\rangle ds\]
 or, in other words, $\mathcal{A}_\Lambda
u_\Lambda\rightharpoonup\mathcal{A}u$ in $L^2(0,T;V')$  and so  $\dot u_\Lambda$
converges weakly  in $L^2(0,T;V')$ by (\ref{problem-discrétisé}).
\par\noindent Thus, letting $\mid\Lambda\mid\to 0$ in (\ref{problem-discrétisé}) shows that
\begin{equation*}
  \dot{u}(t)+\mathcal A(t)u(t)=f(t) \ \ \  t{\rm
-a.e.}  \hbox{ on} \ [0,T],\
 \end{equation*}
 Since $\MR(V,V')\hookrightarrow C([0,T];H),$ we have also that $u_\Lambda\rightharpoonup u$ in $C([0,T];H)$ and in particular  $u_\Lambda(0)\rightharpoonup u(0)$ in $H,$ so that
$u$ satisfies (\ref{nCP in V'}). This completes the proof.
\end{proof}
\section{Invariance of convex sets}\label{s3}
We use the same notations as in the previous sections. We consider a  non-autonomous  form $\fra: [0,T]\times V\times V\rightarrow \C.$
Let $\A(t)\in \L(V,V')$ be the associate operator. In this section we give a other  proof of a known  invariance criterion for  the  non-autonomous homogeneous Cauchy-problem
\begin{equation} \label{nCP in V' homog}
   \dot{u}(t)+\A(t)u(t)=0\ \ t \hbox{-a.e. on}\  [0,T],\ \ \
  \ u(0)=u_0.
\end{equation}

Let $\mathcal C$ be a closed convex subset of the Hilbert space $H$ and let $P: H\rightarrow \mathcal{C}$ be the orthogonal projection onto $\mathcal C;$ i.e. for $x\in H,$ $Px$ is the unique element $x_\mathcal{C}$ in $\mathcal{C}$ such that
\[ \Re(x-x_\mathcal{C}\mid y-x_\mathcal{C})\leq 0 \ \ \hbox{ for all } y\in \mathcal{C}.\]
Recall, that  the closed convex set $\mathcal C$ is invariant for the Cauchy problem $(\ref{nCP in V' homog})$ (in the sense of \cite[Definition 2.1]{ADO12}) if for each $u_0\in \mathcal C$ the solution $u$ of $(\ref{nCP in V' homog})$ satisfies $u(t)\in \mathcal C$ for all $t\in[0,T].$
Recently, Arendt et al. \cite{ADO12} proved that the $\mathcal C$ is invariant for the inhomogenous Cauchy problem $(\ref{nCP in V'})$ provided that $PV\subset V$ and
\[ \Re \fra(t,Pv,v-Pv)\geq \Re \langle f(t),v-Pv\rangle\]
for all $v\in V$ and  for a.e  $t\in[0,T].$
\\\par As consequence of our approach, we obtain easily Theorem 2.2 in \cite{ADO12} for the homogeneous Cauchy problem from  Theorem \ref{alternative-proof}.
\begin{theorem}\label{invariance-charac-f=0}
Let $\fra$ be a non-autonomous form  on $V.$ Let $\mathcal C$ be a closed convex subset of $H.$ Then the convex set $\mathcal C$ is invariant for the Cauchy problem $(\ref{nCP in V' homog})$ provided that $PV\subset V$ and $\Re \fra(t,Pv,v-Pv)\geq 0$ for all $v\in V$ and a.e. $t\in[0,T].$
\end{theorem}
 \begin{proof}
  Let $u_0\in \mathcal C$ and let $u_\Lambda\in \MR(V,V')$ be the solution of $(\ref{nCP in V' homog}).$  The function $u_\Lambda$ is given explicitly  by (\ref{promenade1})-(\ref{promenade2}). From Theorem 2.1 in \cite{Ouh96} (or Theorem 2.2 in \cite{Ouh05}), it follows easily that $u_\Lambda(t)\in \mathcal C$ if and only if $PV\subset V$
  and  \[\Re \fra_k(Pv,v-Pv)\geq 0  \hbox{ for all  } v\in V \hbox{ and } k=0,1,...,n.\]
 Recall that $\fra_k$ is given  by (\ref{eq:form-moyen integrale}). The inequality above holds if and only if $\Re \fra(t,Pv,v-Pv)\geq 0$ for a.e. $t\in[0,T].$  Let now $u$ be the solution of $(\ref{nCP in V' homog}).$ By Theorem \ref{alternative-proof} we have $u_\Lambda\rightharpoonup u$ in
 $\MR(V,V')\underset{d}{\hookrightarrow} C([0,\tau],H).$ The claim follows from the fact that the weak closure of the  convex set $\mathcal C$  is equal to its norm closure.
 \end{proof}
 \begin{theorem}
 Assume that the non-autonomous form $\fra$ is  symmetric and  accretive.  The convex set $\mathcal C$ is invariant for the homogeneous Cauchy problem $(\ref{nCP in V' homog})$ provided that $PV\subset V$  and $\fra(t,Pv,Pv)\leq \fra(t,v,v)$  \hbox{ for a.e.}  $t\in[0,T].$
 \end{theorem}
 \begin{proof}
 Let $u_\Lambda\in \MR(V,V')$ be the solution of $(\ref{nCP in V' homog}).$ By  Theorem 2.2 in \cite{Ouh05}, we have  $u_\Lambda(t)\in \mathcal C$ if and only if $PV\subset V$
  and  \[\fra_k(Pv,Pv)\geq \fra_k(v,v)  \hbox{ for all  } v\in V \hbox{ and } k=0,1,...,n.\]
 This inequality holds  if and only if   $\Re \fra(t,Pv,Pv)\geq \fra(t,v,v)$ for a.e. $t\in[0,T]$ and for all $v\in V.$  The claim follows from the fact that t $u_\Lambda$ converge weakly in $C([0,\tau],H)$ to the solution  of $(\ref{nCP in V' homog}).$ \end{proof}

\section{Well-posedness in $H$}\label{s4}

Recall that  $V,H$ denote two separable Hilbert spaces and $\fra:[0,T]\times V\times V\to \C$ is a non-autonomous  form  introduced in the previous section. We adopt here the notations of  Sections \ref{s2}. We consider the Hilbert space
\[ MR(V,H):= L^2(0,T;V)\cap H^1(0,T; H)\] with norm
\[ \|u\|_{MR(V,H)}^2:=\|u\|^2_{L^2(0,T;V)}+\|u\|^2_{H^1(0,T; H)}.\]
\par\noindent   Let $\Lambda$ be a subdivision of $[0,T]$ and  let $f\in L^2(0,T;H)$ and $u_0\in V.$ The  solution $u_\Lambda$ of  (\ref{problem-discrétisé})
 belongs to  $\MR(V,H) $ and $u_\Lambda\in C([0,T],V).$ In fact, let  $\A_k$ be given  by (\ref{eq:op-moyen integrale}) and let
 $A_k$ be  the part of $\A_k$ in $H.$ Then it is not difficult  to see that
 \begin{equation}\label{eq:par morceaux}u_{\Lambda|_{{[\lambda_k,\lambda_{k+1}[}}}\in MR(\lambda_k,\lambda_{k+1}; D(A_k),H),\ \ \ k=0,1,2,...,n. \end{equation}
 Note, that on each interval $[\lambda_k,\lambda_{k+1}[$ the solution $u_{\Lambda}$ coincides with the solution of the autonomous Cauchy problem
 \begin{equation*}
  \dot{u}_k(t)+ A_ku_k(t)=f(t) \ \  t{\rm
-a.e.} \quad\text on \ (\lambda_k,\lambda_{k+1}),\
 \ \ u_k(\lambda_k)=u_{k-1}(\lambda_k)\in V \
\end{equation*}
which belongs to $MR(\lambda_k,\lambda_{k+1}; D(A_k),H),$ see Section \ref{s1}.
\par\noindent We assume in addition that $\fra$ is \emph{symmetric}; i.e.,
\begin{equation}\label{symmetric}
    \fra(t,u,v) = \overline{\fra(t,v,u)} \quad (t \in [0,T], u,v \in V),
\end{equation}
and \emph{Lipschitz continuous} i.e., there exists a positive constant $L$ such that
\begin{equation}\label{Lipschitz-continuous-form}
    \abs{\fra(t,u,v)- \fra(s,u,v)} \le L \abs{t-s} \norm u_V \norm v_V \quad (t,s \in [0,T], u,v \in V)
\end{equation}
For simplicity, we assume in the following that the subdivision $\Lambda$ of $[0,T]$ is uniform, i.e., $\lambda_{i+1}-\lambda_i=\lambda_{j+1}-\lambda_j$ for all $(i,j)\in\{0,1,2,...,n\}^2.$
\par Theorem \ref{(nCP) in H} below, shows that the solution $u_\Lambda$ of (\ref{problem-discrétisé}) converges weakly in $MR(V,H)$  and so the limit $u,$ which is the solution of (\ref{nCP in V'}), belongs to the maximal regularity space
$ MR(V,H).$ This gives an other proof of
 Theorem 5.1 in \cite{ADLO13} with $\fra$ symmetric and $B=Id.$
{\theorem\label{(nCP) in H}    Assume that  $\fra$ is symmetric and Lipschitz continuous. Let $(f,u_0)\in L^2(0,T; H)\times V.$ Then $u_\Lambda,$ the solution of (\ref{problem-discrétisé}),  converges weakly in $MR(V,H)$ as $|\Lambda|\longrightarrow 0$ and $u:=\lim\limits_{|\Lambda|\to 0}u_\Lambda$
is the unique solution of (\ref{nCP in V'}). Moreover \begin{equation}\label{eq:MR_estimate}
    	\norm u_{\MR(V,H)} \le c \Big[ \norm{u_0}_V + \norm f_{L^2(0,T;H)} \Big],
    \end{equation}
    where the constant $c$ depends merely on $ \alpha,c_H, M$ and $L$.  }
\begin{proof} Let $(f,u_0)\in L^2(0,T; H)\times V.$ Let $u_\Lambda\in MR(V,H)$ be the solution of (\ref{problem-discrétisé}).  According to the proof of Theorem \ref{alternative-proof}, it remains to prove that  $u_\Lambda$ is bounded in $MR(V,H).$ We estimate  first the derivative $\dot u_\Lambda.$ Using  (\ref{rule-formula}) and (\ref{eq:par morceaux}) we obtain
\begin{align*}
\int_0^T\|\dot u_\Lambda(t)\|^2dt&=\int_{0}^{T}\Re(-\A_\Lambda(t)u_\Lambda(t)\mid\dot u_\Lambda(t))dt+
\int_0^T\Re(f(t)\mid\dot u_\Lambda(t)) dt
\\&= \sum_{k=0}^{n-1}\int_{\lambda_{k}}^{\lambda_{k+1}}\Re(-\A_\Lambda(t)u_\Lambda (t)\mid\dot u_\Lambda(t))dt+
\int_0^T\Re(f(t)\mid\dot u_\Lambda(t)) dt
\\&= \sum_{k=0}^{n-1}\int_{\lambda_{k}}^{\lambda_{k+1}}\Re(-\A_ku_\Lambda (t)\mid\dot u_\Lambda(t))dt+
\int_0^T\Re(f(t)\mid\dot u_\Lambda(t)) dt
\\&=  -\sum_{k=0}^{n-1}\int_{\lambda_{k}}^{\lambda_{k+1}}\frac{1}{2}\frac{d}{dt}\fra_k(u_\Lambda(t))dt +
\int_0^T\Re(f(t)\mid\dot u_\Lambda(t)) dt
\end{align*}
 \par\noindent For the first term on the right-hand side of the above equality
 \begin{align*}
 -\sum_{k=0}^{n-1}&\int_{\lambda_{k}}^{\lambda_{k+1}}\frac{d}{dt}\fra_k(u_\Lambda(t))dt
=  -\sum_{k=0}^{n-1}\Big(\fra_k(u_\Lambda(\lambda_{k+1}))-\fra_k(u_\Lambda(\lambda_{k}))\Big)
\\&=-\left(\sum_{k=0}^{n-1}\fra_k(u_\Lambda(\lambda_{k+1}))-\sum_{k=-1}^{n-2}\fra_{k+1}(u_\Lambda(\lambda_{k+1}))\right)
\\&=-\sum_{k=0}^{n-2}\Big(\fra_k(u_\Lambda(\lambda_{k+1}))-\fra_{k+1}(u_\Lambda(\lambda_{k+1}))\Big) -\fra_{n-1}(u_\Lambda(\lambda_{n}))+\fra_{0}(u_\Lambda({0}))
\\&\leq -\sum_{k=0}^{n-2}\Big(\fra_k(u_\Lambda(\lambda_{k+1}))-\fra_{k+1}(u_\Lambda(\lambda_{k+1}))\Big)+M\|u_\Lambda({0})\|_V^2
\end{align*}
Now, using integration by substitution and Lipschitz continuity of  $\fra$ we obtain
\begin{align}\label{b_k} |\fra_k(u_\Lambda(\lambda_{k+1}))-\fra_{k+1}(u_\Lambda(\lambda_{k+1}))|
&\leq L(\lambda_{k+1}-\lambda_{k})\|u_\Lambda(\lambda_{k+1}))\|^2_V
\\\nonumber&\qquad \  \ \ \  \  \ \ \  \hbox{ for every } k=0,1,...,n-2
\end{align}
Let $k=0,1,2,...,n-2$ and $t_k\in [\lambda_k,\lambda_{k+1}[$ be arbitrary. Then $u_{\Lambda|_{{[t_k,\lambda_{k+1}[}}}$ belongs to
$MR(t_k,\lambda_{k+1}; D(A_k),H)$ and
\begin{equation}\label{estimate+integral+riemann}
    \|u_\Lambda(\lambda_{k+1})\|^2_{V}\leq c\Big[\|u_\Lambda(t_k)\|_V^2+\|f\|^2_{L^2(t_k,\lambda_{k+1};H)}\Big]\end{equation}
where the constant $c$ depends only on $M,\omega,  \alpha,c_H$ and $T$ (see Lemma \ref{indepmax}).
Inserting (\ref{estimate+integral+riemann}) into (\ref{b_k}) we obtain then for every $k=0,1,...,n-2$
\begin{align*} |\fra_k(u_\Lambda(\lambda_{k+1}))&-\fra_{k+1}(u_\Lambda(\lambda_{k+1}))|
\\&\leq c(\lambda_{k+1}-\lambda_{k})\|u_\Lambda(t_k)\|_V^2
+c(\lambda_{k+1}-\lambda_{k})\|f\|^2_{L^2(0,T;H)}
\\&\leq c\int_{\lambda_{k}}^{\lambda_{k+1}}\|u_\Lambda(s)\|_V^2ds
+c(\lambda_{k+1}-\lambda_{k})\|f\|^2_{L^2(0,T;H)},
\end{align*}
For the last inequality, $t_k$ is chosen such that $$(\lambda_{k+1}-\lambda_k)\|u_{\Lambda}(t_k)\|_{V}^2=\int_{\lambda_k}^{\lambda_{k+1}}\|u_{\Lambda}(s)\|_{V}^2 ds$$
using the mean value theorem and the fact that $t_k\in [\lambda_k,\lambda_{k+1}[$ is arbitrary.
Thus
    \begin{equation}
    \sum_{k=0}^{n-2}|\fra_k(u_\Lambda(\lambda_{k+1}))-a_{k+1}(u_\Lambda(\lambda_{k+1}))|
\leq c\Big[\|u_\Lambda\|_{L^2(0,T;V)}^2+\|f\|^2_{L^2(0,T;H)}\Big]
 \end{equation}
 for some ${c}={c}( M,\omega,  \alpha,c_H, T,L)$  (possibly different from the previous one).
 It follows
 \begin{align*}
\int_0^T\|\dot u_\Lambda(t)\|^2dt\leq {c} \Big[\|u_\Lambda\|_{L^2(0,T;V)}^2&+\|f\|^2_{L^2(0,T;H)}\Big]
\\&+\int_0^T\Re(f(t)\mid\dot u_\Lambda(t)) dt +M\|u_0\|_V^2
\end{align*}
Finally, from this inequality, the estimate (\ref{estim-lem3}) in Lemma \ref{lem3}, the Cauchy-Schwarz and the Young's inequality applied the third term on right-hand side, it follows that there is a constant $c=c(M,\omega,  \alpha,c_H, T,L)$  such that
\begin{equation*}
\int_0^T\|\dot u_\Lambda(t)\|^2dt+\int_0^T\|u_\Lambda(t)\|_V^2dt\leq c \Big(\|u_0\|_V^2+\|f\|^2_{L^2(0,T;H)}\Big)
\end{equation*}
This completes the proof.
\end{proof}

\emph{Hafida Laasri}, Fachbereich C - Mathematik und Naturwissenschaften, University of Wuppertal, Gaußstraße 20,
42097 Wuppertal, Germany,\\
\texttt{laasri@uni-wuppertal.de}

\quad\\
\noindent
\emph{Ahmed Sani},
Department of Mathematics, University Ibn Zohr, Faculty of Sciences, Agadir, Morocco,
\\ \texttt{ahmedsani82@gmail.com}.

\begin{thebibliography}{999}
\bibitem{Are04} W. Arendt.\ \textit{Semigroups and evolution equations: Functional
calculus, regularity and kernel estimates.} C.M.\ Dafermos,
E.\ Feireisl (Eds.), Handbook of Differential Equations,
Elsevier/North-Holland, Amsterdam (2004), 1-85.
\bibitem{ABHN11}  W.\,Arendt, C.J.K.\ Batty and M.\ Hieber.\ F.\
Neubrander. \textit{Vector-valued Laplace Transforms and Cauchy
Problems.} Birk\"auser Verlag,  Basel, 2011.
\bibitem{ADO12} W.\,Arendt, D.\,Dier and E. M.\,Ouhabaz. \textit{Invariance of convex sets for non-autonomous evolution equations governed by forms.} Available at http://arxiv.org/abs/1303.1167.
\bibitem{ADLO13} W. Arendt, D. Dier, H.~Laasri and E. M.\,Ouhabaz. \textit{Maximal regularity
for evolution equations governed by non-autonomous forms}, Submitted preprint. Available
at http://arxiv.org/abs/1303.1166.
\bibitem{AC10} W.\ Arendt and R.\ Chill. Global existence for quasilinear diffusion equations in isotropic nondivergence form.
\textit{Ann.\ Scuola Norm.\ Sup.\ Pisa CI.\ Sci.\ }(5) Vol.\ IX (2010), 523-539.
\bibitem{Bar71} C.\ Bardos. A regularity theorem for parabolic equations. \textit{J. Functional Analysis,\ } 7 (1971), 311-322.
\bibitem{Bre11} H.\ Br\'ezis. \emph{Functional Analysis, Sobolev Spaces and Partial Differential Equations}.
Springer, Berlin 2011.

\bibitem{DL88} R.\ Dautray and J.L.\  Lions. \textit{Analyse Math\'ematique et
Calcul Num\'erique pour les Sciences et les Techniques.} Vol.\ 8,
  Masson, Paris, 1988.
\bibitem{ELKELA11} O.~El-Mennaoui, V.~Keyantuo, H.~Laasri.
{\it Infinitesimal product of semigroups. Ulmer Seminare.} Heft
16 (2011), 219--230.
\bibitem{ELLA13} O.~El-Mennaoui, H.~Laasri. {\it Stability for non-autonomous linear evolution equations with $L^p-$ maximal regularity.} \textit{ Czechoslovak Mathematical Journal.} 63 (138) 2013.
\bibitem{OH14} B.~Haak, O.~El Maati. \textit{Maximal regulariry for non-autonomous evolution equations.}
 Version available at: http://arxiv.org/abs/1402.1136v1
 \bibitem{Ka} T. Kato. \textit{Perturbation theory for linear operators.} Springer-Verlag, Berlin 1992.
\bibitem{LH} H. Laasri. {\it Problèmes d’évolution et intégrales produits dans les espaces de Banach}. Thèse de Doctorat, Faculté des science Agadir 2012.
\bibitem{Lio61} J.L.\ Lions. \textit{ Equations Diff\'erentielles Op\'erationnelles
 et Probl\`emes aux Limites.} Springer-Verlag, Berlin, G\"ottingen, Heidelberg, 1961.
\bibitem{OS10} E. M.\ Ouhabaz and C.\ Spina. \textit{Maximal regularity
for nonautonomous Schr\"odinger type equations.}  J.\
Differential Equation 248 (2010),1668-1683.
\bibitem{Ouh05} E.\ M.\ Ouhabaz. \textit{Analysis of Heat Equations on Domains}. London Math.\ Soc.\ Monographs,
Princeton Univ.\ Press 2005.
\bibitem{Ouh96} E.\ M.\ Ouhabaz. Invariance of closed convex sets and domination criteria for semigroups. \textit{Pot.\ Analysis}   5 (6) (1996),
611-625.
\bibitem{Paz83} A.\ Pazy. \textit{ Semigroups of Linear Operators and Applications to
Partial Differential Equations.}  Springer-Verlag, Berlin, 1983.
\bibitem{Sho97} R.\ E.\ Showalter.\textit{ Monotone Operators in Banach
Space and Nonlinear Partial Differential Equations.} Mathematical
Surveys and Monographs.  American Mathematical Society, Providence, RI, 1997.
\bibitem{S}{ A.~Slav\'ik}. \textit{Product integration its history and applications}. Matfyzpress, Praha
(2007). Zbl 1216.28001, MR2917851.

\bibitem{Tan79} H. Tanabe. \textit{Equations of Evolution.} Pitman 1979.
\bibitem{Tho03} S. Thomaschewski. \textit{Form Methods for Autonomous and Non-Au\-ton\-o\-mous Cauchy Problems},
PhD Thesis, Universit\"at Ulm 2003.

\end{thebibliography}
\end{document}